\documentclass[11pt]{article} 
\usepackage{amsfonts,amsmath,latexsym,amssymb,mathrsfs,amsthm,comment}
\usepackage{graphicx}
\usepackage{xcolor}
\usepackage{booktabs}

\let\OLDthebibliography\thebibliography
\renewcommand\thebibliography[1]{
  \OLDthebibliography{#1}
  \setlength{\parskip}{1pt}
  \setlength{\itemsep}{0pt plus 0.0ex}
}

\def\numberlikeadb{\global\def\theequation{\thesection.\arabic{equation}}}
\numberlikeadb
\newtheorem{theorem}{Theorem}[section]

\newtheorem{corollary}[theorem]{Corollary}

\newtheorem{remark}[theorem]{Remark}
\newtheorem{example}[theorem]{Example}

\usepackage{lscape}
\usepackage{caption}
\usepackage{multirow}
\allowdisplaybreaks
\begin{document}

\title{On the characteristic function of the asymmetric Student's $t$-distribution and an integral involving the sine function}
\author{Robert E. Gaunt\footnote{Department of Mathematics, The University of Manchester, Oxford Road, Manchester M13 9PL, UK, robert.gaunt@manchester.ac.uk}}

\date{} 
\maketitle

\vspace{-5mm}

\begin{abstract} We obtain a new closed-form formula for the characteristic function of the asymmetric Student's $t$-distribution. As part of our analysis, we derive a new closed-form formula for the integral $\int_0^\infty \sin(ax)/(b^2+x^2)^n\,\mathrm{d}x$, for $a,b>0$, $n\in\mathbb{Z}^+$, expressed in terms of the exponential integral function. As a consequence of our integral formula, we deduce a closed-form formula for the limit $\lim_{\nu\rightarrow n} \{I_{\nu-1/2}(x)-\mathbf{L}_{1/2-\nu}(x)\}/\sin(\pi\nu)$, for $n\in\mathbb{Z}^+$, $x>0$.

\end{abstract}





\section{Introduction}

In this paper, we consider the asymmetric Student's $t$-distribution (AST distribution) that was introduced by 
\cite{zg10}, and represents a generalisation of the skewed version of Student's $t$-distribution given by \cite{fern}. 
For $0<\alpha<1$ and $\nu_1,\nu_2>0$, the probability density function (PDF) of the AST distribution is given by
\begin{align}\label{pdf}
f(x)=\begin{cases}
\displaystyle \frac{\alpha}{\alpha^*}K(\nu_1)\bigg[1+\frac{1}{\nu_1}\Big(\frac{x}{2\alpha^*}\Big)^2\bigg]^{-(\nu_1+1)/2}, &\: x\leq0, \\[10pt]
\displaystyle \frac{1-\alpha}{1-\alpha^*}K(\nu_2)\bigg[1+\frac{1}{\nu_2}\Big(\frac{x}{2(1-\alpha^*)}\Big)^2\bigg]^{-(\nu_2+1)/2}, & \: x>0,
\end{cases}    
\end{align}
where $\alpha^*=\alpha K(\nu_1)/\{\alpha K(\nu_1)+(1-\alpha)K(\nu_2)\}$ and $K(\nu)=\Gamma((\nu+1)/2)/\{\sqrt{\pi\nu}\Gamma(\nu/2)\}$. Here, $\alpha$ is a skewness parameter and $\nu_1$ and $\nu_2$ are left and right tail parameters, respectively. Taking $\alpha=1/2$ and $\nu_1=\nu_2=\nu$ yields the PDF of the classical Student's $t$-distribution with $\nu$ degrees of freedom. The AST distribution has been shown to be a good fit to real financial data (see, for example, 
\cite{li,zg11}) and includes the popular generalisations of
\cite{hansen} and 
\cite{fern} for Student's $t$-distribution as special cases. Indeed, the AST distribution has become a popular alternative to the classical Student's $t$-distribution, as can been from the fact that the works \cite{fern} and \cite{zg10} have attracted around 2000 Google Scholar citations to date. 

In this paper, we will contribute to the distributional theory of the AST distribution by deriving new formulas for its characteristic function. In deriving these formulas, we obtain new closed-form expressions for the integral $\int_0^\infty \sin(ax)/(b^2+x^2)^n\,\mathrm{d}x$, $a,b>0$ in the case $n=2,3,4\ldots$.

A number of the most important distributional properties of the AST distribution were derived by 
\cite{zg10,zg11}. However, a formula for the characteristic function (CF), one of the most fundamental distributional properties, was not provided.
This problem was considered in the work
\cite{n13}, in which closed-form formulas for the CF were given in terms of generalized hypergeometric functions, as well as other special functions including the Bessel, modified Bessel, Struve and modified Struve functions. 
However, the formulas of \cite{n13} are erroneous (indeed, it is clear even just by inspection that the formulas have singularities at $\nu_1=1$ or $\nu_2=1$).


In this paper, we derive a new formula for the CF of the AST distribution (Theorem \ref{thm1}). Our formula takes a simpler form than any of those (incorrect) formulas given by 
\cite{n13}, as it is expressed solely in terms of modified Bessel functions, the modified Struve function of the first kind and the exponential integral, and does not involve generalized hypergeometric functions. Indeed, we believe that our formula is the simplest that can be provided for the AST distribution for general parameters $0<\alpha<1$ and $\nu_1,\nu_2>0$. Our general formula is also rather theoretically satisfying in that it is seen to very easily reduce to the known characteristic function for Student's $t$-distribution when $\alpha=1/2$ and $\nu_1=\nu_2=\nu$ (see Remark \ref{rem1}). 

In deriving our formula for the CF of the AST distribution, we will require a closed-form formula for the integral $\int_0^\infty \sin(ax)/(b^2+x^2)^\rho\,\mathrm{d}x$, for $a,b,\rho>0$. In the case $\rho\in(0,\infty)\setminus\mathbb{Z}^+$, formulas are already available in the literature. Here, $(0,\infty)\setminus\mathbb{Z}^+$ is notation for the set of all positive real numbers except the positive integers. Indeed, a rearrangement of formula 11.5.7 of \cite{olver} yields the formula
\begin{equation}\label{yy}
\int_0^\infty\frac{\sin(ax)}{(b^2+x^2)^\rho}\,\mathrm{d}x=\frac{\sqrt{\pi}}{2}\Gamma(1-\rho)\bigg(\frac{a}{2b}\bigg)^{\rho-1/2}\big\{I_{\rho-1/2}(ab)-\mathbf{L}_{1/2-\rho}(ab)\big\}, 
\end{equation}
which is valid for $a,b>0$ and $\rho\in(0,\infty)\setminus\mathbb{Z}^+$.
Here $I_\nu(x)$ and $\mathbf{L}_\nu(x)$ denote the modified Bessel function of the first kind and the modified Struve function of the first kind (see \cite[Chapters 10 and 11]{olver} for definitions and standard properties). Formula (\ref{yy}) is also given in formula 2.5.6(3) of \cite{integralbook}.
The formula of \cite{integralbook} is stated under the assumption that $\rho>0$; however, the RHS of (\ref{yy}) is undefined in the case $\rho\in\mathbb{Z}^+$. This is because $\Gamma(1-n)$ is undefined for $n\in\mathbb{Z}^+$, whilst $\mathbf{L}_{1/2-n}(x)=I_{n-1/2}(x)$ for $x>0$, $n\in\mathbb{Z}^+$ (see \cite[equation 11.4.4]{olver}).
In the case $\rho=1$, the following formula is available
\begin{align}\label{bom}
\int_0^\infty\frac{\sin(ax)}{b^2+x^2}\,\mathrm{d}x=\frac{1}{2b}\big[\mathrm{e}^{-ab}\mathrm{Ei}(ab)-\mathrm{e}^{ab}\mathrm{Ei}(-ab)\big], \quad a,b>0
\end{align}
(see \cite[equation 3.723(1)]{g07} and \cite[equation 2.5.6(5)]{integralbook}).
To the best of my knowledge there are no results in the literature for $\rho=2,3,4,\ldots$, and such formulas are also not known to \emph{Mathematica}. Here $\mathrm{Ei}$ denotes the exponential integral, which is given, for $x>0$, by $\mathrm{Ei}(x)=\int_{-x}^\infty \mathrm{e}^{-t}/t\,\mathrm{d}t=\int_{-\infty}^x\mathrm{e}^{-t}/t\,\mathrm{d}t$ (with the integrals understood as Cauchy principal value integrals) and $\mathrm{Ei}(-x)=-\int_x^\infty\mathrm{e}^{-t}/t\,\mathrm{d}t$ (see equations 6.2.5 and 6.2.6 of \cite{olver}). In Theorem \ref{thm2}, we fill in this gap in the literature by extending formula (\ref{bom}) to the more general integral $\int_0^\infty \sin(ax)/(b^2+x^2)^n\,\mathrm{d}x$ for all $n\in\mathbb{Z}^+$. Our integral formula complements the integral formula
\begin{equation}\label{bas}
\int_0^\infty\frac{\cos(ax)}{(x^2+b^2)^\rho}\,\mathrm{d}x=\frac{\sqrt{\pi}}{\Gamma(\rho)}\bigg(\frac{a}{2}\bigg)^{\rho-1/2}\frac{K_{\rho-1/2}(ab)}{b^{2\rho-1}}, \quad a,b,\rho>0,   
\end{equation}
which is a simple rearrangement of Basset's integral formula \cite[p.\ 172, Eq.\ (1)]{watson}; Basset earlier derived the formula for $\rho\in\{1/2,3/2,5/3,\ldots\}$ \cite[pp.\ 18--19]{basset}. As a by-product of our analysis, we obtain a closed-form formula for the limit $\lim_{\nu\rightarrow n} \{I_{\nu-1/2}(x)-\mathbf{L}_{1/2-\nu}(x)\}/\sin(\pi\nu)$, for $n\in\mathbb{Z}^+$, $x>0$ (see Corollary \ref{cor}). Our formula complements the well-known limit $\lim_{\nu\rightarrow n} \{I_{-\nu}(x)-I_{\nu}(x)\}/\sin(\pi\nu)=(2/\pi)K_n(x)$, for $n\in\mathbb{Z}$ and $x>0$, where $K_\nu(x)$ is the modified Bessel function of the second kind (see \cite[Chapter 10]{olver}).

\section{Evaluating the integral}

Our formula for the integral is given in the following theorem. Here and elsewhere in this paper we will use the convention that the empty sum is set to zero.

\begin{theorem}\label{thm2} Suppose $a,b>0$ and $n\in\mathbb{Z}^+$. Then
\begin{align}
\int_0^\infty\frac{\sin(ax)}{(b^2+x^2)^n}\,\mathrm{d}x&=\frac{1}{b^{2n-1}}\sum_{k=1}^n\binom{2n-k-1}{n-1}\frac{2^{k-2n}}{(k-1)!}\nonumber\\
&\quad\times\bigg\{(ab)^{k-1}\big[\mathrm{e}^{-ab}\mathrm{Ei}(ab)+(-1)^k\mathrm{e}^{ab}\mathrm{Ei}(-ab)\big]\nonumber\\
&\quad-\sum_{j=1}^{k-1}
(j-1)!\big[1+(-1)^{k+j}\big](ab)^{k-j-1}\bigg\}.\label{one}
\end{align}  
\end{theorem}

\begin{proof} We will prove the result for the case $b=1$; the general case $b>0$ follows by a simple change of variables. We may also restrict to the case $a\in\mathbb{R}\setminus\{0\}$, as formula (\ref{one}) clearly holds for $a=0$.
We begin by noting the following partial fraction decomposition for the rational function $1/(1+x^2)^n$ over the complex numbers, which follows almost immediately from the partial fraction decomposition of $1/(1-x^2)^n$, which is given in Theorem 1 of \cite{v02}. For $x\in\mathbb{R}$ and $n\in\mathbb{Z}^+$ we have the partial fraction decomposition
\begin{align}\label{par}
\frac{1}{(1+x^2)^n}=\sum_{k=1}^n A_k\bigg[\frac{1}{(x+\mathrm{i})^k}+\frac{(-1)^k}{(x-\mathrm{i})^k}\bigg],    
\end{align}
where $A_k=\binom{2n-k-1}{n-1}\mathrm{i}^k/2^{2n-k}$.
With the partial fraction decomposition (\ref{par}) and the standard formula $\sin(x)=(\mathrm{e}^{\mathrm{i}x}-\mathrm{e}^{-\mathrm{i}x})/(2\mathrm{i})$ we can write
\begin{align}
\int_0^\infty\frac{\sin(ax)}{(1+x^2)^n}\,\mathrm{d}x&=\frac{1}{2\mathrm{i}}\sum_{k=1}^n A_k\bigg\{\int_0^\infty\frac{\mathrm{e}^{\mathrm{i}ax}}{(x+\mathrm{i})^k}\,\mathrm{d}x-\int_0^\infty\frac{\mathrm{e}^{-\mathrm{i}ax}}{(x+\mathrm{i})^k}\,\mathrm{d}x\nonumber\\
&\quad+\int_0^\infty\frac{(-1)^k\mathrm{e}^{\mathrm{i}ax}}{(x-\mathrm{i})^k}\,\mathrm{d}x-\int_0^\infty\frac{(-1)^k\mathrm{e}^{-\mathrm{i}ax}}{(x-\mathrm{i})^k}\,\mathrm{d}x\bigg\}. \label{four}  
\end{align}
We now note the following definite integral formula
\begin{align}\label{five}
\int_0^\infty \frac{\mathrm{e}^{-\mu x}}{(x+b)^n}\,\mathrm{d}x=\frac{1}{(n-1)!}\sum_{k=1}^{n-1} (k-1)!(-\mu)^{n-k-1}b^{-k}-\frac{(-\mu)^{n-1}}{(n-1)!}\mathrm{e}^{b\mu}\mathrm{Ei}(-b\mu), 
\end{align}
which is easily derived by repeated integration by parts. Formula (\ref{five}) is stated under the conditions $n\in\mathbb{Z}^+$, $|\mathrm{arg}(b)|<\pi$, $\mathrm{Re}(\mu)>0$ in formulas 3.352(4), 3.353(2) and 3.353(3) of \cite{g07}, and the formula is also valid under the more specific condition that $n\in\mathbb{Z}^+$, $b=c\mathrm{i}$, $\mu=d\mathrm{i}$ for some $c,d\in\mathbb{R}\setminus\{0\}$, which is the case relevant to us. Applying the integral formula (\ref{five}) to equation (\ref{four}) and simplifying yields the formula
\begin{align}
\int_0^\infty\frac{\sin(ax)}{(1+x^2)^n}\,\mathrm{d}x&=\frac{1}{2\mathrm{i}}\sum_{k=1}^n \frac{A_k}{(k-1)!}\bigg\{\sum_{j=1}^{k-1}(j-1)!(\mathrm{i}a)^{k-j-1}\mathrm{i}^{-j}-(\mathrm{i}a)^{k-1}\mathrm{e}^{a}\mathrm{Ei}(-a)\nonumber\\
&\quad-\sum_{j=1}^{k-1}(j-1)!(-\mathrm{i}a)^{k-j-1}\mathrm{i}^{-j}-(-\mathrm{i}a)^{k-1}\mathrm{e}^{-a}\mathrm{Ei}(a)\nonumber\\
&\quad+(-1)^k\sum_{j=1}^{k-1}(j-1)!(\mathrm{i}a)^{k-j-1}(-\mathrm{i})^{-j}-(-1)^k(\mathrm{i}a)^{k-1}\mathrm{e}^{-a}\mathrm{Ei}(a)\nonumber\\
&\quad-(-1)^k\sum_{j=1}^{k-1}(j-1)!(-\mathrm{i}a)^{k-j-1}(-\mathrm{i})^{-j}-(-1)^k(-\mathrm{i}a)^{k-1}\mathrm{e}^{a}\mathrm{Ei}(-a)\bigg\}\nonumber\\
&=\sum_{k=1}^n\binom{2n-k-1}{n-1}\frac{2^{k-2n}}{(k-1)!}\bigg\{a^{k-1}\big[\mathrm{e}^{-a}\mathrm{Ei}(a)+(-1)^k\mathrm{e}^a\mathrm{Ei}(-a)\big] \nonumber\\
&\quad+\sum_{j=1}^{k-1}(-1)^{k+j+1}(j-1)!\big[1+(-1)^{k+j}\big]a^{k-j-1}\bigg\}.\label{two}
\end{align} 
The desired formula (\ref{one}) now follows from (\ref{two}) by making the substitution $x=y/b$ in the integral, and a final simplification that $(-1)^{k+j+1}[1+(-1)^{k+j}]=-[1+(-1)^{k+j}]$ for $j,k\in\mathbb{Z}^+$.
\end{proof}

\begin{example}In the cases $n=2,3,4$, formula (\ref{one}) yields the following expressions:
\begin{align*}
\int_0^\infty\frac{\sin(ax)}{(b^2+x^2)^2}\,\mathrm{d}x&=\frac{1}{4b^3}\big[(1+ab)\mathrm{e}^{-ab}\mathrm{Ei}(ab)-(1-ab)\mathrm{e}^{ab}\mathrm{Ei}(-ab)\big], \\ 
\int_0^\infty\frac{\sin(ax)}{(b^2+x^2)^3}\,\mathrm{d}x&=\frac{1}{16b^5}\Big[\big(3+3ab+(ab)^2\big)\mathrm{e}^{-ab}\mathrm{Ei}(ab)-\big(3-3ab+(ab)^2\big)\mathrm{e}^{ab}\mathrm{Ei}(-ab)-2ab\Big], \\
\int_0^\infty\frac{\sin(ax)}{(b^2+x^2)^4}\,\mathrm{d}x&=\frac{1}{96b^7}\Big[\big(15+15ab+6(ab)^2+(ab)^3\big)\mathrm{e}^{-ab}\mathrm{Ei}(ab)\\
&\quad\quad\quad\quad-\big(15-15ab+6(ab)^2-(ab)^3\big)\mathrm{e}^{ab}\mathrm{Ei}(-ab)-14ab\Big].
\end{align*}
\end{example}

\begin{corollary}\label{cor}Let $x>0$ and $n\in\mathbb{Z}^+$. Then
\begin{align}
\lim_{\nu\rightarrow n} \frac{I_{\nu-1/2}(x)-\mathbf{L}_{1/2-\nu}(x)}{\sin(\pi\nu)}&=\frac{2}{\pi^{3/2}}(n-1)!\bigg(\frac{2}{x}\bigg)^{n-1/2}\sum_{k=1}^n\binom{2n-k-1}{n-1}\frac{2^{k-2n}}{(k-1)!}\nonumber\\
&\quad\times\bigg\{x^{k-1}\big[\mathrm{e}^{-x}\mathrm{Ei}(x)+(-1)^k\mathrm{e}^{x}\mathrm{Ei}(-x)\big]\nonumber\\
&\quad-\sum_{j=1}^{k-1}(j-1)!\big[1+(-1)^{k+j}\big]x^{k-j-1}\bigg\}.\label{limit}
\end{align}
\end{corollary}

\begin{proof}From the relation $\Gamma(x)\Gamma(1-x)=\pi/\sin(\pi x)$ (see \cite[equation 5.5.3]{olver}) and the integral formula (\ref{yy}) we obtain that, for $x>0$ and $\nu>0$,
\begin{align*}
\frac{I_{\nu-1/2}(x)-\mathbf{L}_{1/2-\nu}(x)}{\sin(\pi\nu)}=\frac{\Gamma(\nu)}{\pi}\cdot \frac{2}{\sqrt{\pi}}\bigg(\frac{2}{x}\bigg)^{\nu-1/2}\int_0^\infty\frac{\sin(xt)}{(1+t^2)^\nu}\,\mathrm{d}t,   
\end{align*}
so that, for $n\in\mathbb{Z}^+$, 
\begin{align*}
\lim_{\nu\rightarrow n}\frac{I_{\nu-1/2}(x)-\mathbf{L}_{1/2-\nu}(x)}{\sin(\pi\nu)}=\frac{2}{\pi^{3/2}}(n-1)!\bigg(\frac{2}{x}\bigg)^{n-1/2}\int_0^\infty\frac{\sin(xt)}{(1+t^2)^n}\,\mathrm{d}t.  
\end{align*}
Evaluating the integral using formula (\ref{one}) now yields the limit (\ref{limit}).
\end{proof}

\section{The characteristic function of the AST distribution}

Our formula for the CF of the AST distribution is given in the following theorem. 
We let $\mathrm{sgn}(x)$ denote the sign function, which is given by $\mathrm{sgn}(x)=-1$ for $x<0$, $\mathrm{sgn}(0)=0$, $\mathrm{sgn}(x)=1$ for $x>0$.


\begin{theorem}\label{thm1} Let $X$ be a random variable with PDF (\ref{pdf}), with $0<\alpha<1$ and $\nu_1,\nu_2>0$. Denote its CF by $\phi_X(t)=\mathbb{E}[\mathrm{e}^{\mathrm{i}tX}]$. Then, for $t\in\mathbb{R}$,
\begin{align}\label{gen}
\phi_X(t)=\mathcal{A}_{\alpha,\nu_1}(t)+ \mathcal{A}_{1-\alpha,\nu_2}(t) +\mathrm{i}\{\mathcal{B}_{1-\alpha,\nu_2}(t)-\mathcal{B}_{\alpha,\nu_1}(t)\},
\end{align}
where
\begin{align}
\mathcal{A}_{\alpha,\nu}(t)&=\frac{2\alpha(\alpha^*\sqrt{\nu}|t|)^{\nu/2}K_{\nu/2}(2\alpha^*\sqrt{\nu}|t|)}{\Gamma(\nu/2)}, \quad\nu>0, \nonumber\\
\mathcal{B}_{\alpha,\nu}(t)&=\frac{\pi\alpha(\alpha^*\sqrt{\nu}|t|)^{\nu/2}\mathrm{sgn}(t)}{\cos(\pi\nu/2)\Gamma(\nu/2)}\big\{I_{\nu/2}(2\alpha^*\sqrt{\nu}|t|)-\mathbf{L}_{-\nu/2}(2\alpha^*\sqrt{\nu}|t|)\big\}, \:\,\nu\in(0,\infty)\setminus\{1,3,5,\ldots\}, \label{mar1} \\
\mathcal{B}_{\alpha,\nu}(t)&=\frac{2\alpha}{\sqrt{\pi}}\frac{\Gamma(\nu+1/2)}{\Gamma(\nu/2)}\mathrm{sgn}(t)\sum_{k=1}^{(\nu+1)/2}\binom{\nu-k}{(\nu-1)/2}\frac{2^{k-\nu-1}}{(k-1)!}\nonumber\\
&\quad\times\bigg\{(2\alpha^*\sqrt{\nu}|t|)^{k-1}\big[\mathrm{e}^{-2\alpha^*\sqrt{\nu}|t|}\mathrm{Ei}(2\alpha^*\sqrt{\nu}|t|)+(-1)^k\mathrm{e}^{2\alpha^*\sqrt{\nu}|t|}\mathrm{Ei}(-2\alpha^*\sqrt{\nu}|t|)\big]\nonumber\\
&\quad-\sum_{j=1}^{k-1}
(j-1)!\big[1+(-1)^{k+j}\big](2\alpha^*\sqrt{\nu}|t|)^{k-j-1}\bigg\}, \quad\nu\in\{1,3,5,\ldots\}. \label{mar2}
\end{align}
\end{theorem}

\begin{remark}
The expression for $\mathcal{B}_{\alpha,\nu}(t)$ given in formula (\ref{mar1}) is undefined for $\nu\in\{1,3,5,\ldots\}$ (due to the factor $\cos(\pi\nu/2)$ in the denominator). It is for this reason that an alternative expression is needed for the case $\nu\in\{1,3,5,\ldots\}$, which is provided by formula (\ref{mar2}).   
\end{remark}

\begin{proof} We prove the result for the case $t>0$; the case $t<0$ is similar and hence omitted. By Euler's formula, we can write $\phi_X(t)=\mathbb{E}[\mathrm{e}^{\mathrm{i}tX}]=\mathbb{E}[\cos(tX)]+\mathrm{i}\mathbb{E}[\sin(tX)]$, so that, on using that $\sin(x)$ is an odd function in $x$ and that $\cos(x)$ and $(1+cx^2)^{-\rho}$ are even functions in $x$, we see that the CF of $X$ can be expressed as
\begin{align*}
\phi_X(t)&=\mathcal{I}_{\alpha,\nu_1}(t)+ \mathcal{I}_{1-\alpha,\nu_2}(t) +\mathrm{i}\{\mathcal{J}_{1-\alpha,\nu_2}(t)-\mathcal{J}_{\alpha,\nu_1}(t)\},   
\end{align*}
where
\begin{align}
\mathcal{I}_{\alpha,\nu}(t)&=\int_0^\infty \frac{\alpha}{\alpha^*}K(\nu)\bigg[1+\frac{1}{\nu}\Big(\frac{x}{2\alpha^*}\Big)^2\bigg]^{-(\nu+1)/2}\cos(tx)\,\mathrm{d}x \nonumber \\  
&=2\alpha\sqrt{\nu} K(\nu)\int_0^\infty (1+y^2)^{-(\nu+1)/2}\cos\big(2\alpha^*\sqrt{\nu}ty\big)\,\mathrm{d}y \label{ieqn}
\end{align}
and
\begin{align}
\mathcal{J}_{\alpha,\nu}(t)&=\int_0^\infty \frac{\alpha}{\alpha^*}K(\nu)\bigg[1+\frac{1}{\nu}\Big(\frac{x}{2\alpha^*}\Big)^2\bigg]^{-(\nu+1)/2}\sin(tx)\,\mathrm{d}x\nonumber\\
&=2\alpha\sqrt{\nu} K(\nu)\int_0^\infty (1+y^2)^{-(\nu+1)/2}\sin\big(2\alpha^*\sqrt{\nu}ty\big)\,\mathrm{d}y.\label{jeqn}
\end{align}

For the integral $\mathcal{I}_{\alpha,\nu}(t)$, we note that from the integral formula (\ref{bas}) we have
the definite integral formula
\begin{equation}\label{xx}
\int_0^\infty (1+u^2)^{-(v+1/2)}\cos(xu)\,\mathrm{d}u=\frac{\sqrt{\pi}x^vK_v(x)}{2^v\Gamma(v+1/2)}, \quad x>0,\,v>-1/2.    
\end{equation}
Evaluating the integral in (\ref{ieqn}) using the integral formula (\ref{xx}) reveals that $\mathcal{I}_{\alpha,\nu}(t)=\mathcal{A}_{\alpha,\nu}(t)$.

In the case $\nu\in(0,\infty)\setminus\{1,3,5,\ldots\}$, evaluating the integral in (\ref{jeqn}) using the integral formula (\ref{yy}) together with the formula $\Gamma((\nu+1)/2)\Gamma((1-\nu)/2)=\pi/\cos(\pi\nu/2)$ (an easy consequence of equation 5.5.3 of \cite{olver}), confirms that $\mathcal{J}_{\alpha,\nu}(t)=\mathcal{B}_{\alpha,\nu}(t)$ for $\nu\in(0,\infty)\setminus\{1,3,5,\ldots\}$. For the case $\nu\in\{1,3,5,\ldots\}$, we evaluate the integral in (\ref{jeqn}) using the integral formula (\ref{one}), which confirms that $\mathcal{J}_{\alpha,\nu}(t)=\mathcal{B}_{\alpha,\nu}(t)$ for $\nu\in\{1,3,5,\ldots\}$.
This completes the proof of formula (\ref{gen}).    
\end{proof}


\begin{remark}\label{rem1} 
Setting $\alpha=1/2$ and $\nu_1=\nu_2=\nu$ in formula (\ref{gen}) yields the formula
\begin{equation*}
\phi_X(t)=2\mathcal{A}_{1/2,\nu}(t)=\frac{(\sqrt{\nu}|t|)^{\nu/2}K_{\nu/2}(\sqrt{\nu}|t|)}{2^{\nu/2-1}\Gamma(\nu/2)},   \quad t\in\mathbb{R}, 
\end{equation*} 
which we recognise as the known formula for the CF of Student's $t$-distribution with $\nu$ degrees of freedom (see, for example, \cite{gaunt}).
\end{remark}

    


A simple generalisation of the AST distribution that incorporates a location parameter $\mu\in\mathbb{R}$ and scale parameter $\sigma>0$ was also introduced by 
\cite{zg10}, for which the PDF is given by
\begin{align}\label{nine}
f(y)=\begin{cases}
\displaystyle \frac{1}{\sigma}\bigg[1+\frac{1}{\nu_1}\Big(\frac{y-\mu}{2\alpha\sigma K(\nu_1)}\Big)^2\bigg]^{-(\nu_1+1)/2}, &\: y\leq\mu, \\[10pt]
\displaystyle \frac{1}{\sigma}\bigg[1+\frac{1}{\nu_2}\Big(\frac{y-\mu}{2(1-\alpha)\sigma K(\nu_2)}\Big)^2\bigg]^{-(\nu_2+1)/2}, & \: y>\mu.
\end{cases}    
\end{align}
In the following corollary, we present a formula for the CF of this distribution, which follows easily from Theorem \ref{thm1}.

\begin{corollary}Let $Y$ be a random variable with PDF (\ref{nine}), with $0<\alpha<1$, $\nu_1,\nu_2>0$, $\mu\in\mathbb{R}$ and $\sigma>0$. Denote its CF by $\phi_Y(t)=\mathbb{E}[\mathrm{e}^{\mathrm{i}tY}]$. Then, for $t\in\mathbb{R}$,
\begin{align}\label{gen0}
\phi_Y(t)=\mathrm{e}^{\mathrm{i}\mu t}\big(\mathcal{A}_{\alpha,\nu_1}(\sigma Bt)+ \mathcal{A}_{1-\alpha,\nu_2}(\sigma Bt) +\mathrm{i}\{\mathcal{B}_{1-\alpha,\nu_2}(\sigma Bt)-\mathcal{B}_{\alpha,\nu_1}(\sigma Bt)\}\big),
\end{align}
where $B=\alpha K(\nu_1)+(1-\alpha)K(\nu_2)$, and $\mathcal{A}_{\alpha,\nu}(t)$ and $\mathcal{B}_{\alpha,\nu}(t)$ are defined as in Theorem \ref{thm1}.    
\end{corollary}

\begin{proof} We begin by noting the following series of equalities that were given by \cite[equation (3)]{zg10}:
\begin{align*}
\frac{\alpha}{\alpha^*}K(\nu_1)=\frac{1-\alpha}{1-\alpha^*}K(\nu_2)=\alpha K(\nu_1)+(1-\alpha)K(\nu_2)=B.    
\end{align*}
From these equalities we can infer the distributional relation $Y=_d \mu+\sigma BX$, where $X$ is a random variable with PDF (\ref{pdf}) and $=_d$ denotes equality in distribution. Formula (\ref{gen0}) now follows from the relation $\phi_Y(t)=\mathrm{e}^{\mathrm{i}\mu t}\phi_X(\sigma B t)$ and formula (\ref{gen}) for $\phi_X(t)$.
\end{proof}

\section*{Acknowledgements}
The author was funded in part by EPSRC grant EP/Y008650/1 and EPSRC grant UKRI068. I would like to thank the reviewers for their helpful comment and suggestions.


\footnotesize

\begin{thebibliography}{99}
\addcontentsline{toc}{section}{References}



\bibitem{basset} Basset, A. B. \emph{A Treatise on Hydrodynamics; With Numerous Examples, Volume II.} Cambridge. {Deighton}, {Bell} and {Co}. {London}. {George} {Bell} and {Sons}, 1888.

\bibitem{fern} Fernandez, C. and Steel, M. F. J. On Bayesian modelling of fat tails and skewness. \emph{J. Am. Stat. Assoc.} $\mathbf{93}$ (1998), 359--371.

\bibitem{gaunt} Gaunt, R. E. A simple proof of the characteristic function of Student's $t$-distribution.  \emph{Commun. Stat. Theory} $\mathbf{50}$ (2021), 3380--3383.

\bibitem{g07} Gradshteyn, I. S. and Ryzhik, I. M.  \emph{Table of Integrals, Series and Products,}  $7$th ed.  Academic Press, 2007.

\bibitem{hansen} Hansen, B. E. Autoregressive conditional density estimation. \emph{Int. Econ. Rev.} $\mathbf{35}$ (1994), 705--730.

\bibitem{li} Li, R. and Nadarajah, S. A review of Student's $t$ distribution and its generalizations. \emph{Empir. Econ.} $\mathbf{58}$ (2020), 1461--1490.

\bibitem{n13} Nadarajah, S., Chan, S. and Afuecheta, E., 2013. On the characteristic function for asymmetric Student $t$ distributions. \emph{Econ. Lett.} $\mathbf{121}$ (2013), 271--274.


\bibitem{olver} Olver, F. W. J., Lozier, D. W., Boisvert, R. F. and Clark, C. W.  \emph{NIST Handbook of Mathematical Functions.} Cambridge University Press, 2010.

\bibitem{integralbook} Prudnikov, A. B., Brychkov, Y. A. and Marichev, O. I. \emph{Integrals and Series: Volume 1.}  CRC Press, 1992. 

\bibitem{v02} Velleman, D. J. Partial Fractions, Binomial Coefficients and the Integral of an Odd Power of $\sec \theta$. \emph{Am. Math. Mon.} $\mathbf{109}$ (2002), 746--749.

\bibitem{watson} Watson, G. N. {\it A Treatise on the Theory of Bessel Functions.} Cambridge University Press, 1922.

\bibitem{zg10} Zhu, D. and Galbraith, J. W. A generalized asymmetric Student-$t$ distribution with application to financial econometrics. \emph{J. Econometrics} $\mathbf{157}$ (2010), 297--305.

\bibitem{zg11} Zhu, D. and Galbraith, J. W. Modeling and forecasting expected shortfall with the generalized asymmetric Student-$t$ and asymmetric exponential power distributions. \emph{J. Empir. Finac.} $\mathbf{18}$ (2011), 765--778.

\end{thebibliography}
\end{document}